\documentclass[11pt,twoside]{amsart}
\usepackage{latexsym,amssymb,amsmath}
\usepackage[all]{xy}

\numberwithin{equation}{section}
\newtheorem{theorem}{Theorem}[section]
\newtheorem{lemma}[theorem]{Lemma}
\newtheorem{proposition}[theorem]{Proposition}
\newtheorem{corollary}[theorem]{Corollary}

\theoremstyle{definition}
 
\newtheorem{procedure}[theorem]{Procedure} 
\newtheorem{remark}[theorem]{Remark}
\newtheorem{example}[theorem]{Example}

\begin{document}


\title
{Generalized Hamming weights of projective Reed--Muller-type codes
over graphs}

\thanks{The first and third authors were supported by SNI, Mexico. Second author was
supported by a scholarship from CONACYT, Mexico}

\author[J. Mart\'\i nez-Bernal]{Jos\'e Mart\'\i nez-Bernal}
\address{
Departamento de
Matem\'aticas\\
Centro de Investigaci\'on y de Estudios Avanzados del IPN\\
Apartado Postal
14--740 \\
07000 Mexico City, Mexico.
}
\email{jmb@math.cinvestav.mx}

\author[M. A. Valencia-Bucio]{Miguel A. Valencia-Bucio}
\address{
Departamento de
Matem\'aticas\\
Centro de Investigaci\'on y de Estudios
Avanzados del
IPN\\
Apartado Postal
14--740 \\
07000 Mexico City, Mexico
}
\email{mavalencia@math.cinvestav.mx}

\author[R. H. Villarreal]{Rafael H. Villarreal}
\address{
Departamento de
Matem\'aticas\\
Centro de Investigaci\'on y de Estudios
Avanzados del
IPN\\
Apartado Postal
14--740 \\
07000 Mexico City, Mexico
}
\email{vila@math.cinvestav.mx}

\begin{abstract} 
Let $G$ be a connected graph and let $\mathbb{X}$ be the
set of projective points defined by the column vectors of the
incidence matrix of $G$ over a field $K$ of any characteristic. 
We determine the generalized Hamming weights of the
Reed--Muller-type code over the set $\mathbb{X}$ in terms of graph
theoretic invariants. As an application to coding theory we show that 
 if $G$ is non-bipartite and $K$ is a finite field of ${\rm
 char}(K)\neq 2$, then the 
$r$-th generalized Hamming weight of the linear code generated by the
rows of the incidence matrix of $G$ is the $r$-th weak edge
biparticity of $G$. If ${\rm char}(K)=2$ or $G$ is bipartite, we prove
that the $r$-th generalized Hamming weight of that code is the 
$r$-th edge connectivity of $G$.
\end{abstract}

\maketitle 

\section{Introduction}\label{intro-section}
In this work we study basic parameters of projective Reed--Muller-type codes 
over graphs using an algebraic geometric approach 
via graph theory and commutative algebra, and give some applications 
to linear codes whose generator matrices are incidence matrices of
graphs.  

Let $K$ be a field of characteristic $p\geq 0$, let $G$ be a
connected graph with vertex 
set $V(G)$ and edge
set $E(G)$, and let $t_1,\ldots,t_s$ and $f_1,\ldots,f_m$ be 
the vertices and edges of $G$, respectively. The \textit{incidence
matrix} of $G$, over the field $K$, is the $s\times m$ matrix $A=(a_{ij})$ given by 
$a_{ij}=1$ if $t_i\in f_j$ and $a_{ij}=0$ otherwise. 
The \textit{edge biparticity} of $G$, denoted $\varphi(G)$, is the
minimum number of edges whose removal makes the graph bipartite, and
maybe not connected.  
The $r$-th \textit{weak edge biparticity} of $G$, denoted $\upsilon_r(G)$, is the
minimum number of edges whose removal results in a graph with $r$ 
bipartite components, and maybe some non-bipartite components. If $r=1$, $\upsilon_1(G)$ is the 
\textit{weak edge biparticity} of $G$ and is denoted by $\upsilon(G)$. 
The $r$-th \textit{edge connectivity} of $G$, denoted $\lambda_r(G)$,
is the minimum number of edges whose removal results in a graph with
$r+1$ connected components. If $r=1$, $\lambda_1(G)$ is the
edge connectivity of $G$ and is denoted by $\lambda(G)$. We will use these
invariants to study the minimum distance and the Hamming
weights of Reed-Muller-type codes over graphs. 

The edge biparticity and the edge connectivity are well studied
invariants of a graph
\cite{Har,Zaslavsky}. 
In Section~\ref{RMT-codes-section} we give an algebraic method for computing 
the edge biparticity (Proposition~\ref{saturday-afternoon-dec8-18}). 
For a discussion of computational and algorithmic aspects of edge
bipartization problems we refer to \cite{Pilipczuk-Pilipczuk}. 

The set of columns $\{P_1,\ldots,P_m\}$ of $A$ can be regarded as a
set of points 
$\mathbb{X}=\{[P_1],\ldots,[P_m]\}$ in a
projective space $\mathbb{P}^{s-1}$ over the field $K$. Consider a
polynomial ring $S=K[t_1,\ldots,t_s]=\bigoplus_{d=0}^\infty S_d$ 
over the field $K$ with the standard grading. The 
\textit{vanishing ideal} $I(\mathbb{X})$ of $\mathbb{X}$ is the
graded ideal of $S$ generated by the homogeneous polynomials of $S$ 
that vanish at all points of $\mathbb{X}$. Fix integers $d\geq 1$
and $r\geq 1$. 
The aim of this work is
to determine the following number in terms of the 
combinatorics of the graph $G$:
$$
\delta_\mathbb{X}(d,r):=\min\{|\mathbb{X}\setminus
V_\mathbb{X}(F)|:\, F=\{f_i\}_{i=1}^r\subset S_d,\,
\dim_K(\{\overline{f}_i\}_{i=1}^r)=r\},
$$
where $V_\mathbb{X}(F)$ is the set of zeros or
projective variety of $F$ in $\mathbb{X}$, and
$\overline{f}_i=f_i+I(\mathbb{X})$ is the class of $f_i$ modulo
$I(\mathbb{X})$. This is equivalent to
determine
$$
\text{hyp}_\mathbb{X}(d,r):=\max\{|V_\mathbb{X}(F)|:\, F=\{f_i\}_{i=1}^r\subset S_d,\,
\dim_K(\{\overline{f}_i\}_{i=1}^r)=r\}
$$
because
$\delta_\mathbb{X}(d,r)=|\mathbb{X}|-\text{hyp}_\mathbb{X}(d,r)$. 

A  \textit{projective Reed--Muller-type code} of degree $d$ on 
$\mathbb{X}$ \cite{duursma-renteria-tapia,GRT}, denoted $C_\mathbb{X}(d)$, is the image of the 
following evaluation linear map 
\begin{align*}
\quad {\rm ev}_d\colon S_d\rightarrow K^{m},&\quad f\mapsto
\left(f(P_1),\ldots,f(P_m)\right).
\end{align*}

\quad The motivation to study $\delta_\mathbb{X}(d,r)$ comes from algebraic
coding theory because---over a finite field---the $r$-th generalized Hamming 
weight of the Reed--Muller-type code $C_\mathbb{X}(d)$ of degree $d$ is equal 
to $\delta_\mathbb{X}(d,r)$ \cite[Lemma~4.3(iii)]{rth-footprint}.

Generalized Hamming weights were introduced by Helleseth, Kl{\o}ve and
Mykkeltveit \cite{helleseth,klove} and were first used systematically
by Wei \cite{wei}. For convenience we recall this notion. Let $K=\mathbb{F}_q$ be a finite
field and let $C$ be a 
$[m,k]$ {\it linear
code} of {\it length} $m$ and {\it dimension} $k$, 
that is, $C$ is a linear subspace of $K^m$ with $k=\dim_K(C)$. Let $1\leq r\leq k$ be an integer.  
Given a linear subspace $D$ of $C$, the \textit{support} of
$D$ is the set   
$$
\chi(D):=\{i\,\vert\, \exists\, (a_1,\ldots,a_m)\in D,\, a_i\neq 0\}.
$$
\quad The $r$-th \textit{generalized Hamming weight} of $C$, denoted
$\delta_r(C)$, is given by
$$
\delta_r(C):=\min \{|\chi(D)|\colon D \mbox{ is a subspace of }C,\, \dim_K (D)=r\}.
$$
\quad 
The set $\{\delta_1(C), \ldots, \delta_k(C)\}$ is called the
\textit{weight hierarchy} of the code $C$. The following
\textit{duality} of Wei \cite[Theorem~3]{wei} is a classical result in
this area that shows a strong
relationship between the weight hierarchies of $C$ and its dual
$C^\perp$:
$$
\{\delta_i(C)\,\vert\,i=1,\ldots,k\}=\{1,\ldots,m\}\setminus\{m+1-
\delta_i(C^\perp)\,\vert\, i=1,\ldots,m-k\}.
$$
\quad These numbers are
a natural generalization of the notion of minimum 
distance and they have several applications from cryptography (codes
for wire--tap channels of type II), $t$--resilient functions, trellis
or branch complexity of linear codes, and shortening or puncturing
structure of codes; see
\cite{GHWCartesian,carvalho,ghorpade,geil,rth-footprint,GHW2014,
Pellikaan,Johnsen,olaya,schaathun-willems,tsfasman,
wei,wei-yang,Yang} and the
references therein. If $r=1$, we obtain the
\textit{minimum distance} $\delta(C)$ 
of $C$ which is the most important parameter of a linear code. 
In this paper 
we give combinatorial formulas for the weight hierarchy of
$C_\mathbb{X}(d)$ for $d\geq 1$.   

Our main results are: 

\noindent {\bf Theorems~\ref{pepe-vila-2018},
\ref{pepe-vila-2018-char=2}, \ref{pepe-vila-2018-hybrid}}\textit{ Let $G$ 
be a connected graph with $s$ vertices, $m$ edges, $r$-th weak edge biparticity
$\upsilon_r(G)$, $r$-th edge connectivity  
$\lambda_r(G)$, and let
$A$ be the incidence matrix of $G$ over a field $K$ of characteristic
$p$. If $\mathbb{X}$ is the set of
column vectors of $A$, then  
$$ 
\delta_\mathbb{X}(d,r)=\delta_r(C_\mathbb{X}(d))
=\begin{cases}
\upsilon_r(G)&\text{if }d=1,\, p\neq 2,\, G\textit{ is non-bipartite}, 1\leq r\leq s,\\
\lambda_r(G)&\text{if }d=1,\, p=2,\, 1\leq r\leq s-1,\\
\lambda_r(G)&\text{if }d=1,\,\textit{G is bipartite}, 1\leq r\leq s-1,\\
r&\text{if } d\geq 2\text{ and }1\leq r\leq m.\\
\end{cases}
$$
}
\quad Thus computing $\upsilon_r(G)$ and $\lambda_r(G)$ is equivalent to
computing the $r$-th generalized Hamming weight of $C_\mathbb{X}(1)$  
for $K=\mathbb{F}_2$ or $K=\mathbb{F}_3$. These are the only cases
that matter. 

The \textit{incidence matrix code} of a graph $G$ over a finite field
$K$ of characteristic $p$, denoted $C_p(G)$, is the
linear code generated by the rows of the incidence matrix of
$G$. 
As an application to coding theory we obtain the following combinatorial
formulas for the generalized Hamming weights of $C_p(G)$ when $G$ is
connected (Corollary~\ref{pepe-vila-2018-coro1}). 
$$ 
\delta_r(C_p(G))=\begin{cases}
\upsilon_r(G)&\text{if } p\neq 2,\, G\textit{ is non-bipartite}, 1\leq r\leq s,\\
\lambda_r(G)&\text{if }p=2,\, 1\leq r\leq s-1,\\
\lambda_r(G)&\text{if }G\textit{ is bipartite}, 1\leq r\leq s-1.
\end{cases}
$$
\quad The \textit{minimum distance} of the incidence matrix code of
the graph $G$ is defined as
$$
\delta(C_p(G)):=\min\{\omega(a)\colon a \in C_p(G) \setminus
\{0\}\},
$$ 
where $\omega(a)$ is the Hamming weight of the vector $a$, that
is, the number of non-zero entries of $a$. The 
minimum distance of $C_p(G)$ is $\delta_1(C_p(G))$, the $1$st Hamming
weight of this code. Then we can recover the combinatorial formulas of Dankelmann, Key and Rodrigues 
\cite[Theorems~1--3]{Dankelmann-Key-Rodrigues} for the minimum
distance of $C_p(G)$ in terms of the weak edge biparticity
$\upsilon(G)$ and the
edge connectivity  $\lambda(G)$ of $G$ (Corollary~\ref{Dankelmann-etal}). 

Using Macaulay~2 \cite{mac2}, 
SageMath \cite{sage}, and Wei's duality \cite[Theorem~3]{wei}, we can
compute the weight hierarchy of $C_p(G)$. In
Sections~\ref{examples-section} and \ref{procedures-section}, we illustrate this with some
examples and procedures. There are algebraic methods that can be used to obtain lower bounds for
$\delta_r(C_p(G))$ or equivalently for $\lambda_r(G)$ and
$\upsilon_r(G)$ \cite[Theorem~4.9]{rth-footprint}.

\smallskip

For all unexplained
terminology and additional information  we refer to 
\cite{Boll,diestel,Har} (for graph theory),
\cite{MacWilliams-Sloane,tsfasman} 
(for the theory of
error-correcting codes and linear codes), and
\cite{Eisen,cocoa-book,monalg-rev} (for 
commutative algebra and Hilbert functions). 

\section{Reed--Muller-type codes over connected
graphs}\label{RMT-codes-section}

In this section we present our main results. 
To avoid repetitions, we continue to employ
the notations and 
definitions used in Section~\ref{intro-section}.

\begin{lemma}\label{nov19-18} Let $G$ be a connected graph and let $e_1,\ldots,e_r$ be a minimum
set of edges whose removal makes the graph bipartite. Then there is
$\omega\colon V(G)\rightarrow\{+,-\}$ such that the edges of $G$ 
whose vertices have the same sign are precisely
$e_1,\ldots,e_r$. 
\end{lemma}

\begin{proof} If $G$ is bipartite, there is nothing to prove. If $G$
is non-bipartite, pick a bipartition $V_1$, $V_2$ of the 
graph $G\setminus\{e_1,\ldots,e_r\}$. Setting $\omega(v)=+$ if 
$v\in V_1$ and $\omega(v)=-$ if $v\in V_2$, note that the vertices
of each $e_i$ have the same sign. Indeed if the vertices of $e_i$ have
different sign, then 
$G\setminus\{e_1,\ldots,e_{i-1},e_{i+1},\ldots,e_r\}$ is bipartite, 
a contradiction. 
\end{proof}

The edge biparticity of a graph $G$ can be 
easily expressed by considering all possible ways of making $G$ a
vertex-signed graph. 

\begin{lemma}\label{dec7-18} Let $G$ be a connected graph, let $\mathcal{F}$ be the
set of surjective maps $\omega\colon V(G)\rightarrow\{+,-\}$,
and let $E_\omega$ be the set of edges of $G$ whose vertices have the
same sign. Then
$$
\varphi(G)=\min\{|E_\omega|:\, \omega\in\mathcal{F}\}.
$$
\end{lemma}

\begin{proof} If $G$ is bipartite, $\varphi(G)=0$ and there is nothing
to prove. Assume that $G$ is non-bipartite. Then
$E_\omega\neq\emptyset$ for $\omega\in\mathcal{F}$. By
Lemma~\ref{nov19-18}, there is $\omega\in\mathcal{F}$ such that
$\varphi(G)=|E_\omega|$. Thus, one has the inequality ``$\geq$''. To
show the reverse inequality take 
$\omega$ in $\mathcal{F}$. It suffices to show that
$\varphi(G)\leq|E_\omega|$. The vertex set of $G$ can be
partitioned as $V(G)=V^+\cup V^-$, where $V^+$ (resp. $V^-$) is the set
of vertices of $G$ with positive (resp. negative) sign. Then
$G\setminus E_\omega$ is bipartite with bipartition $V^+$, $V^-$. 
Thus $\varphi(G)\leq|E_\omega|$.
\end{proof}

This lemma can be used to compute $\varphi(G)$.
Let $K$ be a field of ${\rm char}(K)\neq 2$. 
Each $\omega$ in $\mathcal{F}$ defines a linear polynomial 
$$
h_\omega=\sum_{\omega(t_i)=+}t_i-\sum_{\omega(t_i)=-}t_i.
$$
\quad The number of points of $\mathbb{X}$ where $h_\omega$ does not
vanish is equal to $|E_\omega|$. As a consequence one obtains the
following algebraic formula for the edge biparticity.

\begin{proposition}\label{saturday-afternoon-dec8-18} 
Let $G$ be a connected non-bipartite graph over a
field of ${\rm char}(K)\neq 2$. Then
$$
\varphi(G)=\min\{|\mathbb{X}\setminus V_\mathbb{X}(h)|\colon
\, h=a_1t_1+\cdots+a_st_s,\, a_i\in\{1,-1\},\ \forall\, i\}.
$$
\end{proposition}

\begin{proof} Any $h=a_1t_1+\cdots+a_st_s$, $a_i\in\{1,-1\}$ for all
$i$, $h\neq\pm(t_1+\cdots+t_s)$, can be written as $h=h_\omega$ for
some $\omega\in\mathcal{F}$. As $|E_\omega|=|\mathbb{X}\setminus
V_\mathbb{X}(h_\omega)|$ for $\omega\in\mathcal{F}$, the result
follows from Lemma~\ref{dec7-18}. 
\end{proof}

This result can be used in practice to compute $\varphi(G)$ using
\textit{Macaulay}$2$ \cite{mac2} (see the examples and procedures of 
Sections~\ref{examples-section} and \ref{procedures-section}).

\begin{remark}\label{dec8-18} If we allow $a_1,\ldots,a_s$ to be in $\{0,1,-1\}$
such that not all of them are zero, we obtain the minimum distance of
$C_p(G)$. This follows from \cite[Lemma~4.3(iii)]{rth-footprint}. 
\end{remark}

The following result is well known. 

\begin{proposition}{\cite{bjorner-rank-incidence,Kulk1,Key-codes}}\label{rank-incidence} 
Let $G$ be a connected  graph with $s$ vertices and
let $A$ be its incidence matrix over a field $K$. Then 
$$
{\rm rank}(A)=\begin{cases} s&\text{if }{\rm char}(K)\neq 2\mbox{ and
 }G\textit{ is non-bipartite},\\
s-1&\text{if }{\rm char}(K)=2 \text{ or }G\text{ is bipartite}.
\end{cases}
$$
\end{proposition}

\begin{corollary}\label{dim-cp} Let $G$ be a connected graph with $s$ vertices and 
$m$ edges and let $C=C_p(G)$ $($resp. $C^\perp)$ be the
code $($resp. dual code$)$ of $G$. Then
\begin{itemize}
\item[(a)] $C$ $($resp. $C^\perp$$)$ is an $[m,s]$ $($resp. $[m,m-s]$$)$ code if $p\neq 2$
and $G$ is non-bipartite. 
\item[(b)] $C$ $($resp. $C^\perp$$)$ is an $[m,s-1]$ $($resp.
$[m,m-s+1]$$)$ code if $p=2$ or $G$ is bipartite. 
\end{itemize}
\end{corollary}

\begin{proof} This follow from Proposition~\ref{rank-incidence}
noticing that $\dim(C)+\dim(C^\perp)=m$.
\end{proof}

\begin{lemma}\label{nov23-18} Let $G$ be a connected graph and let $K$
be a field. The following hold.

{\rm (a)} If ${\rm char}(K)\neq 2$, $G$ is non-bipartite and $h$ is
a linear form in $I(\mathbb{X})$, then
$h=0$.

{\rm (b)} If ${\rm char}(K)=2$ and $h\neq 0$ is
a linear form in $I(\mathbb{X})$, then
$h=c\sum_{i=1}^st_i$, for some $c\in K$.

{\rm (c)} If ${\rm char}(K)=2$ and $h$ is
a linear form in $I(\mathbb{X})$ in $s-1$ variables, then
$h=0$.
\end{lemma}

\begin{proof} Let $\psi$ be 
the linear map $\psi\colon K^s\rightarrow K^m$, $x\mapsto xA$. Fix 
a linear from $h=\sum_{i=1}^sa_it_i$ of $S_1$ and set
$v_h=(a_1,\ldots,a_s)$. Then $v_h$ is in ${\rm ker}(\psi)$ if 
and only if  $h\in I(\mathbb{X})$. For use below notice 
that  $s=\dim({\rm ker}(\psi))+{\rm rank}(A)$.

(a): By Proposition~\ref{rank-incidence}, ${\rm ker}(\psi)=(0)$. Thus
$v_h=0$, that is, $h=0$.

(b): From Proposition~\ref{rank-incidence} 
we get that $\ker(\psi)$ has dimension $1$, and since $\mathbf{1} =
(1,\ldots,1)\in\ker(\psi)$ the result follows.

(c): It is a consequence of (b).
\end{proof}

\begin{lemma}\label{nov24-18} Let $G$ be a
connected bipartite graph with bipartition $V_1$, $V_2$. The following
hold.
\begin{itemize}
\item[(a)] If $K$ is a field and $h\neq 0$ is
a linear form of $S$ that vanishes at all points of $\mathbb{X}$, 
then $h=c(\sum_{t_i\in V_1}t_i- \sum_{t_i\in V_2}t_i)$ for 
some $c\in K$.
\item[(b)] If $t_i$ and $t_j$ are in $V_1$, then $G\cup\{t_i,t_j\}$
contains an odd cycle.
\end{itemize}
\end{lemma}

\begin{proof} (a): It follows adapting the proof of
Lemma~\ref{nov23-18}.

(b): As $G$ is connected and bipartite, there is a path $\mathcal{P}$ in $G$ of even length
joining $t_i$ and $t_j$. Then, adding the new edge $\{t_i,t_i\}$ to
the path $\mathcal{P}$, gives an odd cycle of $G\cup\{t_i,t_j\}$. 
\end{proof}

\begin{lemma}\label{nov28-18} Let $G$ be a connected non-bipartite graph. If 
$\ell=\upsilon_r(G)$ and $f_1,\ldots,f_\ell$ are edges of $G$, then the graph
$H=G\setminus\{f_1,\ldots,f_\ell\}$ has at most $r$ bipartite
connected components. 
\end{lemma}

\begin{proof} Let $H_1,\ldots,H_n$ be the connected components of $H$.
We proceed by contradiction assuming that $H_1,\ldots,H_{r+1}$ are
bipartite. Consider the graph $G'=G\setminus\{f_2,\ldots,f_\ell\}$. If
$f_1\subset V(H_i)$ for some $i$, then $G'$ has $r$ bipartite components,
a contradiction. Thus $f_1\not\subset V(H_i)$ for $i=1,\ldots,n$. 
Hence, $f_1$ joins $H_i$ and $H_j$ for some $i,j$ with $i<j$. If
$j\leq r+1$, the graph $H_i\cup H_j\cup\{f_1\}$ is bipartite and
connected, and $G'$ has $r$ bipartite components, a contradiction.
Thus $j>r+1$ and in this case $G'$ has $r$ bipartite components, a
contradiction. 
\end{proof}

We come to one of our main results.

\begin{theorem}\label{pepe-vila-2018} Let $G$ be a connected
non-bipartite graph with $s$
vertices and $m$ edges, let $K$ be a field of ${\rm char}(K)\neq 2$, and let
$A$ be the incidence matrix of $G$. If $\mathbb{X}$ is the set of
column vectors of $A$ and $\upsilon_r(G)$ is the $r$-th weak edge biparticity 
of $G$, then 
$$ 
\delta_\mathbb{X}(d,r)
=\begin{cases}
\upsilon_r(G)&\text{
if }d=1\text{ and }1\leq r\leq s=\dim_K(C_\mathbb{X}(d)),\\
r&\text{ if } d\geq 2\text{ and }1\leq r\leq m=\dim_K(C_\mathbb{X}(d)).\\
\end{cases}
$$
\end{theorem}

\begin{proof} Assume $d=1$. First we show the inequality
$\delta_\mathbb{X}(1,r)\geq\upsilon_r(G)$. We proceed by contradiction
assuming that $\upsilon_r(G)>\delta_\mathbb{X}(1,r)$. Then 
$\upsilon_r(G)>|\mathbb{X}\setminus V_\mathbb{X}(F)|$ for some 
set $F$ consisting of $r$ linear forms $h_1,\ldots,h_r$ which are
linearly independent, over $K$, modulo $I(\mathbb{X})$. Let $[P_1],\ldots,[P_\ell]$ be the
points in $\mathbb{X}\setminus V_\mathbb{X}(F)$ and let
$f_1,\ldots,f_\ell$ be the edges of $G$ corresponding
to these points. Consider the graph
$H=G\setminus\{f_1,\ldots,f_\ell\}$. Let $H_1,\ldots,H_n$ be the
bipartite connected components of $H$. Since $\upsilon_r(G)>\ell$, $n$ is at
most $r-1$. Let $\mathbb{X}_H$ be the set of points corresponding to the columns of the incidence
matrix of $H$. Note that $h_i$ vanishes at all points of 
$\mathbb{X}_H$ for $i=1,\ldots,r$. Then, by
Lemma~\ref{nov23-18}, $h_1,\ldots,h_r$ are linear forms in the
variables $V(H_1)\cup\cdots\cup V(H_n)$. For each $1\leq j\leq n$,
let $A_1^j$, $A_2^j$ be the
bipartition of $H_j$ and set $g_j=\sum_{t_i\in A_1^j}t_i- \sum_{t_i\in A_2^j}t_i$. Then,
by Lemma~\ref{nov24-18}, $F=\{h_1,\ldots,h_r\}$ is in the $K$-linear space generated by
$g_1,\ldots,g_n$, a contradiction because $F$ is linearly independent
over $K$ and $n<r$. 

Now we show the inequality $\delta_\mathbb{X}(1,r)\leq\upsilon_r(G)$.
Note that, by Lemma~\ref{nov23-18}, any minimal generator of
$I(\mathbb{X})$ has degree at least $2$. Hence, it
suffices to find a set $F=\{h_1,\ldots,h_r\}$ of linearly independent
forms of degree $1$ such that
$\upsilon_r(G)=|\mathbb{X}\setminus V_\mathbb{X}(F)|$. We set
$\ell=\upsilon_r(G)$. There are edges $f_1,\ldots,f_\ell$ of $G$ such
that the graph
$$
H=G\setminus\{f_1,\ldots,f_\ell\}
$$
has exactly $r$ connected bipartite components (see
Lemma~\ref{nov28-18}). We denote the connected
components of $H$ by $H_1,\ldots,H_n$, where $H_1,\ldots,H_r$ are bipartite.
Consider a bipartition $A_1^j$, $A_2^j$ of $H_j$ for $j=1,\ldots,r$ and set 
$$h_j=\sum_{t_i\in A_1^j} t_i-\sum_{t_i\in A_2^j} t_i.$$
\quad Let $P_i$ be the point in $\mathbb{P}^{s-1}$ that corresponds to
$f_i$ for $i=1,\ldots,\ell$. To complete the proof of the case $d=1$ we need only show 
the equality $\{[P_1],\ldots,[P_\ell]\}=|\mathbb{X}\setminus
V_\mathbb{X}(F)|$. To show the inclusion ``$\subset$'' 
fix an edge $f_k$ with $1\leq k\leq \ell$ and set
$$
H'=\bigcup_{i=1}^rH_i,\quad 
H''=\bigcup_{i=r+1}^nH_i\ \text{ and }\
G'=G\setminus\{f_1,\ldots,f_{k-1},f_{k+1},\ldots,f_\ell\}. 
$$
\quad Note that $f_k\not\subset V(H_j)$ for $r< j$, otherwise $G'$ 
has $r$ bipartite components. As a consequence $f_k$ intersects $V(H')$, 
otherwise $f_k\subset V(H'')$, $f_k$ joins $H_i$ and $H_j$ for some
$r<i<j$, and the graph $G'$ has $r$ bipartite components, a
contradiction. 

Case (1): $f_k\subset V(H_j)$ for some $1\leq j\leq r$. As
$V(H_j)=A_1^j\cup A_2^j$, either $f_k\subset A_1^j$ or 
$f_k\subset A_2^j$, otherwise the graph $G'$ has $r$ bipartite
components, a contradiction. Hence, as ${\rm char}(K)\neq 2$, we get that
$h_j(P_k)\neq 0$. Thus $[P_k]\in\mathbb{X}\setminus V_\mathbb{X}(F)$.

Case (2): $f_k\cap V(H_i)\neq\emptyset$ and $f_k\cap
V(H_j)\neq\emptyset $
for some $i<j\leq r$. Then using the bipartitions of $H_i$ and $H_j$
we get $h_i(P_k)\neq 0$ and $h_j(P_k)\neq 0$. 
Thus
$[P_k]\in\mathbb{X}\setminus V_\mathbb{X}(F)$. 

Case (3): $f_k\cap V(H_i)\neq\emptyset$
for some $1\leq i\leq r$ and $f_k\cap V(H'')\neq\emptyset$. 
Then using the bipartition of $H_i$ we get $h_i(P_k)\neq 0$. Thus
$[P_k]\in\mathbb{X}\setminus V_\mathbb{X}(F)$. 

To show the inclusion ``$\supset$'' take 
$[P]\in\mathbb{X}\setminus V_\mathbb{X}(F)$ and denote by $f$ its
corresponding edge in $G$. Then there is $1\leq j\leq n$ such that
$h_j(P)\neq 0$. We proceed by contradiction assuming
$[P]\notin\{[P_1],\ldots,[P_\ell]\}$, that is, $f\neq f_i$ for
$i=1,\ldots\ell$. Then $f$ is an edge of $H$. Thus $f$ is an edge of
$H_k$ for some $1\leq k\leq n$. If $r<k$, then $h_i(P)=0$ for
$i=1,\ldots,r$ by construction of the $h_i$'s, a contradiction. Thus
$1\leq k\leq r$. If $f\subset A_1^k$ or $f\subset A_2^k$, then $f$
would not be an edge of $H_k$, a contradiction. Hence $f$ joins $A_1^k$ with
$A_2^k$, and consequently $h_i(P)=0$ for $i=1,\ldots,r$ by
construction of the $h_i$'s, a contradiction. Thus
$P=P_i$ for some $1\leq i\leq \ell$, as required.

Assume $d\geq 2$. We
claim that in this case the evaluation function ${\rm ev}_d$ is
surjective. 
Indeed, taking all $m$-tuples of the form ${\rm ev}_d(t_it_j^{d-1})$, where
$\{t_i,t_j\}$ is an edge of $G$, one gets the canonical basis of
$K^m$. Therefore $C_\mathbb{X}(d) = K^m$ and
$\delta_\mathbb{X}(d,r)=r$ for $1\leq r\leq m$.
\end{proof}

We come to another of our main results.

\begin{theorem}\label{pepe-vila-2018-char=2} Let $G$ be a connected graph with $s$
vertices and $m$ edges, let $K$ be a field of ${\rm char}(K)=2$, and let
$A$ be the incidence matrix of $G$. If $\mathbb{X}$ is the set of
column vectors of $A$ and $\lambda_r(G)$ is the $r$-th edge
connectivity of $G$, then 
$$ 
\delta_\mathbb{X}(d,r)
=\begin{cases}
\lambda_r(G)&\text{
if }d=1\text{ and }1\leq r\leq s-1=\dim_K(C_\mathbb{X}(d)),\\
r&\text{ if } d\geq 2\text{ and }1\leq r\leq m=\dim_K(C_\mathbb{X}(d)).\\
\end{cases}
$$
\end{theorem}

\begin{proof} Assume $d=1$. First we show the inequality
$\delta_\mathbb{X}(1,r)\geq\lambda_r(G)$. We proceed by contradiction
assuming that $\lambda_r(G)>\delta_\mathbb{X}(1,r)$. Then 
$\lambda_r(G)>|\mathbb{X}\setminus V_\mathbb{X}(F)|$ for some 
set $F$ consisting of $r$ linear forms $h_1,\ldots,h_r$ which are
linearly independent modulo $I(\mathbb{X})$. We set
$\ell=|\mathbb{X}\setminus V_\mathbb{X}(F)|$. Let $[P_1],\ldots,[P_\ell]$ be the
points in $\mathbb{X}\setminus V_\mathbb{X}(F)$ and let
$f_1,\ldots,f_\ell$ be the edges of $G$ corresponding
to these points. Consider the graph
$H=G\setminus\{f_1,\ldots,f_\ell\}$ and denote by $H_1,\ldots,H_n$ its
connected components. Since $\lambda_r(G)>\ell$, $H$ cannot have $r+1$
components, that is, $n\leq r$. Let $\mathbb{X}_H$ be the set of 
points corresponding to the columns of the incidence
matrix of $H$. Note that $h_i$ vanishes at all points of 
$\mathbb{X}_H$ for $i=1,\ldots,r$. Indeed, take a point $[P]$ in
$\mathbb{X}_H$, then its corresponding edge $f$ is in $H_k$ for 
some $k$, then $f\neq f_j$ for $j=1,\ldots,\ell$, that is, 
$[P]\notin \mathbb{X}\setminus V_\mathbb{X}(F)$. Thus $h_i(P)=0$. We set 
$g_j=\sum_{t_i\in V(H_j)}t_i$ for $j=1,\ldots,n$. As $h_i\in
I(\mathbb{X}_H)$, by Lemma~\ref{nov23-18}, $h_i$ is a linear combination
of $g_1,\ldots,g_n$ for $i=1,\ldots,r$. Therefore
$$
Kh_1\oplus\cdots\oplus Kh_r\subset Kg_1\oplus\cdots\oplus Kg_n,
$$
and consequently $r\leq n$. Thus $r=n$ and the inclusion above is an
equality. Therefore taking classes modulo $I(\mathbb{X})$, we get
$$
K\overline{h}_1\oplus\cdots\oplus K\overline{h}_r=
K\overline{g}_1\oplus\cdots\oplus K\overline{g}_n.
$$
\quad As $\overline{h}_1,\ldots,\overline{h}_r$ are linearly
independent, so are $\overline{g}_1,\ldots,\overline{g}_n$ 
because $r=n$, a contradiction because by construction of the $g_i$'s 
and since ${\rm char}(K)=2$, 
one has $\sum_{i=1}^n\overline{g}_i=\sum_{i=1}^s\overline{t}_i=\overline{0}$. 

Next we show the inequality $\delta_\mathbb{X}(1,r)\leq\lambda_r(G)$.
It suffices to find a set 
$F=\{h_1,\ldots,h_r\}$ of forms of degree $1$ whose image
$\overline{F}=\{\overline{h}_1,\ldots,\overline{h}_r\}$ in 
$S/I(\mathbb{X})$ is linearly independent over $K$ and 
$\lambda_r(G)=|\mathbb{X}\setminus V_\mathbb{X}(F)|$. We set
$\ell=\lambda_r(G)$. There are edges $f_1,\ldots,f_\ell$ of $G$ such
that the graph
$$
H=G\setminus\{f_1,\ldots,f_\ell\}
$$
has exactly $r+1$ connected components $H_1,\ldots,H_{r+1}$. For
$j=1,\ldots,r$, we set 
$$
h_j=\sum_{t_i\in V(H_j)}t_i.
$$
\quad Note that $h_i$ and $h_j$ have no common variables for $i\neq j$
and any sum of the polynomials $h_1,\ldots,h_r$ is a
linear form in $s-1$ variables. Hence, by Lemma~\ref{nov23-18}(b),
$\overline{F}$ is linearly independent. 

Let $[P_i]$ be the point in $\mathbb{P}^{s-1}$ that corresponds to 
$f_i$ for $i=1,\ldots,\ell$. To complete the proof of the case $d=1$ we need only show 
the equality $\{[P_1],\ldots,[P_\ell]\}=|\mathbb{X}\setminus
V_\mathbb{X}(F)|$. To show the inclusion ``$\subset$'' 
fix an edge $f_k$ with $1\leq k\leq \ell$ and set
$$
G'=G\setminus\{f_1,\ldots,f_{k-1},f_{k+1},\ldots,f_\ell\}. 
$$
\quad Note that $f_k\not\subset V(H_j)$ for $j=1,\ldots,r+1$, otherwise $G'$ 
has $r+1$ components, a contradiction. As a consequence $f_k$ joins
$H_i$ and $H_j$ for some $i<j$. Thus $h_i(P_k)\neq 0$ and
$[P_k]\in\mathbb{X}\setminus V_\mathbb{X}(F)$.

To show the inclusion ``$\supset$'' take 
$[P]\in\mathbb{X}\setminus V_\mathbb{X}(F)$ and denote by $f$ its
corresponding edge in $G$. Then there is $1\leq j\leq r$ such that
$h_j(P)\neq 0$. We proceed by contradiction assuming
$[P]\notin\{[P_1],\ldots,[P_\ell]\}$, that is, $f\neq f_i$ for
$i=1,\ldots\ell$. Then $f$ is an edge of $H$. As ${\rm char}(K)=2$, we
get $h_i(P)=0$ for $i=1,\ldots,r$ by construction of $h_i$, a
contradiction.

If $d\geq 2$, the equality $\delta_\mathbb{X}(d,r)= r$ for 
$1\leq r\leq m=\dim_K(C_\mathbb{X}(d))$ follows from the proof 
of Theorem~\ref{pepe-vila-2018}.
\end{proof}

The next result is a hybrid of Theorems~\ref{pepe-vila-2018} 
and \ref{pepe-vila-2018-char=2} and is characteristic free.

\begin{theorem}\label{pepe-vila-2018-hybrid} Let $G$ be a connected
bipartite graph with $s$ vertices and $m$ edges, let $K$ be a field
of any characteristic, and let
$A$ be the incidence matrix of $G$. If $\mathbb{X}$ is the set of
column vectors of $A$ and $\lambda_r(G)$ is the $r$-th edge
connectivity of $G$, then 
$$ 
\delta_\mathbb{X}(d,r)
=\begin{cases}
\lambda_r(G)&\text{
if }d=1\text{ and }1\leq r\leq s-1=\dim_K(C_\mathbb{X}(d)),\\
r&\text{ if } d\geq 2\text{ and }1\leq r\leq m=\dim_K(C_\mathbb{X}(d)).\\
\end{cases}
$$
\end{theorem}

\begin{proof} Let $V_1$, $V_2$ be the bipartition of $G$. 
Consider the set $\mathbb{Y}$ of all points 
$[\mathbf{e}_i-\mathbf{e}_j]$ in $\mathbb{P}^{s-1}$ such that $\{t_i,t_j\}$ is an edge
of $G$ with $t_i\in V_1$ and $t_j\in V_2$, 
where $\mathbf{e}_i$ is the $i$-th unit vector in $K^s$. Noticing 
that the polynomial $h=t_1+\cdots+t_s$ vanishes at all points of
$\mathbb{Y}$ and the equality $C_\mathbb{X}(1)=C_\mathbb{Y}(1)$,
 the result follows adapting  Lemma~\ref{nov23-18} and the proof of 
Theorem~\ref{pepe-vila-2018-char=2} with $\mathbb{Y}$ playing the role
of $\mathbb{X}$. 
\end{proof}
The main application to coding theory is the following.
\begin{corollary}\label{pepe-vila-2018-coro1} Let $C_p(G)$ be the
incidence matrix code of a connected graph
$G$ with $s$ vertices, $m$ 
edges, $r$-th weak edge biparticity
$\upsilon_r(G)$, $r$-th edge connectivity  
$\lambda_r(G)$, over a finite field $K$ of ${\rm
char}(K)=p$. Then the $r$-th generalized Hamming weight of $C_p(G)$ is given by 
$$ 
\delta_r(C_p(G))=\begin{cases}
\upsilon_r(G)&\text{if } p\neq 2,\, G\textit{ is non-bipartite}, 1\leq r\leq s,\\
\lambda_r(G)&\text{if }p=2,\, 1\leq r\leq s-1,\\
\lambda_r(G)&\text{if }G\textit{ is bipartite}, 1\leq r\leq s-1.
\end{cases}
$$
\end{corollary}

\begin{proof} Note that the linear code $C_p(G)$ is the image of $S_1$---the vector space of linear
forms of $S$---under the evaluation map ${\rm ev}_1\colon
S_1\rightarrow K^{m}$, $f\mapsto\left(f(P_1),\ldots,f(P_m)\right)$. 
The image of the linear function $t_i$, under the
map $\text{ev}_1$, gives the $i$-th row of the incidence 
matrix of $G$. This means that $C_p(G)$ is the Reed--Muller-type code
$C_\mathbb{X}(1)$. Hence, the result follows using the equality 
$\delta_\mathbb{X}(1,r) =
\delta_r(C_p(G))$  and Theorems~\ref{pepe-vila-2018},
\ref{pepe-vila-2018-char=2}, and \ref{pepe-vila-2018-hybrid}.
\end{proof}
As a consequence we recover the following result.
\begin{corollary}{\cite[Theorems~1--3]{Dankelmann-Key-Rodrigues}}\label{Dankelmann-etal} 
Let $C_p(G)$ be the incidence matrix code of a connected graph
$G$ with $s$ vertices, $m$ 
edges, weak edge biparticity
$\upsilon(G)$, edge connectivity  
$\lambda(G)$, over a finite field $K$ of ${\rm
char}(K)=p$. Then the minimum distance of $C_p(G)$ is given by 
$$ 
\delta(C_p(G))=\begin{cases}
\upsilon(G)&\text{if } p\neq 2,\, G\textit{ is non-bipartite}, 1\leq r\leq s,\\
\lambda(G)&\text{if }p=2,\, 1\leq r\leq s-1,\\
\lambda(G)&\text{if }G\textit{ is bipartite}, 1\leq r\leq s-1.
\end{cases}
$$
\end{corollary}

\begin{proof} It follows from Corollary~\ref{pepe-vila-2018-coro1}
making $r=1$.
\end{proof}

\section{Examples}\label{examples-section}
Let $G$ be a connected graph and let $C_p(G)$ be the 
incidence matrix code of $G$ over a finite field $K=\mathbb{F}_q$ of
characteristic $p$. As an application of our main results we can use
\textit{Macaulay}$2$ \cite{mac2}, 
SageMath \cite{sage}, and Wei's duality \cite[Theorem~3]{wei}, to 
compute the weight hierarchy of $C_p(G)$. Hence, by
Corollary~\ref{pepe-vila-2018-coro1}, we can compute 
the corresponding higher weak biparticity and edge connectivity numbers of
the graph. We do not claim, however, to have found an efficient way to
compute the weight hierarchy of an incidence matrix code.

Conversely any algorithm that computes these graph invariants
can be used to compute the weight hierarchy of $C_p(G)$. Using
Proposition~\ref{saturday-afternoon-dec8-18}, we can also compute the
edge biparticity of $G$ using the field of rational numbers.

The weight hierarchy of $C_p(G)$ can also be computed using
a formula of Johnsen and Verdure \cite{JohVer} for the
Hamming weights in terms of the Betti numbers of the
Stanley--Reisner ring whose faces are the independent sets of the
vector matroid of a parity check matrix of $C_p(G)$.

We illustrate how to use our results in practice with some examples.

\begin{example}\label{example-graph1} Let $G$ be the graph of Figure~\ref{graph1}. Recall
that the dimension of $C_p(G)$ is $6$ if $p=3$ and is $5$ if $p=2$ 
(Corollary~\ref{dim-cp}). For use below we denote the 
dual code by $C_p(G)^\perp$.

\begin{figure}[ht]
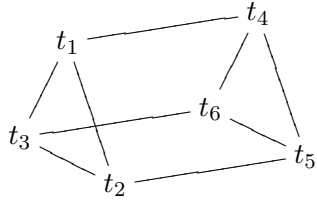

\begin{displaymath}
\xygraph{
 !{<0cm,0cm>;<1cm,0cm>;<0cm,1cm>}
 !{(0,.5)}*+{\text{$t_3$}}="v3"
 !{(.5,1.5)}*+{\text{$t_1$}}="v1"
 !{(1,0)}*+{\text{$t_2$}}="v2"
 !{(2,.8)}*+{\text{$t_6$}}="v6"
 !{(2.5,1.8)}*+{\text{$t_4$}}="v4"
 !{(3,.3)}*+{\text{$t_5$}}="v5"
 "v1"-"v2" "v2"-"v3" "v1"-"v3"
 "v4"-"v5" "v1"-"v4" "v2"-"v5"
 "v3"-@{-}"v6" "v5"-@{-}"v6" "v6"-@{-}"v4"
 }
\end{displaymath}
\caption{Non-bipartite graph $G$.}\label{graph1}
\end{figure}

Using
Procedure~\ref{procedure-example-graph1}, together with, Wei's
duality \cite[Theorem~3]{wei} we obtain Table~\ref{table1} with the weight hierarchy of
$C_p(G)$. The edge biparticity of this graph is $2$, the weak edge
biparticity is $2$, and the edge connectivity is $3$.

\begin{table}[h]
\begin{eqnarray*}
\hspace{-11mm}&&\left.
\begin{array}{c|c|c|c|c|c|c}
r & 1 & 2 & 3& 4&5&6 \\
   \hline
\delta_r(C_2(G))& 3 & 5 & 6 & 8& 9 
 \\ 
  \hline  \delta_r(C_2(G)^\perp)
 &3 &6 &8  &9 & \\ 
   \hline
\delta_r(C_3(G))\,    & 2 & 4 &5  &7 &8 &9
 \\   
\hline
 \delta_r(C_3(G)^\perp) &4& 7& 9 & &\\ 
\end{array}
\right.
\end{eqnarray*}
\caption{Weight hierarchy of $C_p(G)$ for the graph of
Figure~\ref{graph1}.}\label{table1}
\end{table}
\end{example}

\begin{example}\label{example-petersen} Let $G$ be the Petersen graph of
Figure~\ref{petersen}. Recall
that the dimension of $C_p(G)$ (resp. $C_p(G)^\perp$) is $9$ (resp.
$6$) if $p=2$, and the dimension of $C_p(G)$ (resp. $C_p(G)^\perp$)
is $10$ (resp.
$5$) if $p\neq 2$ (Corollary~\ref{dim-cp}). 
\begin{figure}[ht]
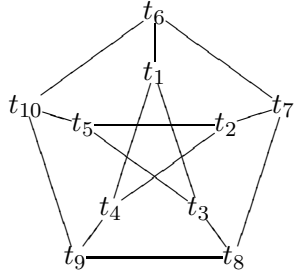

$$
 \xygraph{
 !{<0cm,0cm>:<0cm,1cm>:<-1cm,0cm>::}
 !{(0,0);a(0)**{}?(1)}*{\text{$t_1$}}="a1"
 !{(0,0);a(72)**{}?(1)}*{\text{$t_5$}}="a2"
 !{(0,0);a(144)**{}?(1)}*{\text{$t_4$}}="a3"
 !{(0,0);a(216)**{}?(1)}*{\text{$t_3$}}="a4"
 !{(0,0);a(288)**{}?(1)}*{\text{$t_2$}}="a5"
 !{(0,0);a(0)**{}?(1.8)}*{\text{$t_6$}}="b1"
 !{(0,0);a(72)**{}?(1.8)}*{\text{$t_{10}$}}="b2"
 !{(0,0);a(144)**{}?(1.8)}*{\text{$t_9$}}="b3"
 !{(0,0);a(216)**{}?(1.8)}*{\text{$t_8$}}="b4"
 !{(0,0);a(288)**{}?(1.8)}*{\text{$t_7$}}="b5"
 "a1"-"a3" "a3"-"a5" "a5"-"a2" "a2"-"a4" "a4"-"a1"
 "b1"-"b2" "b2"-"b3" "b3"-"b4" "b4"-"b5" "b5"-"b1"
 "a1"-"b1" "a2"-"b2" "a3"-"b3" "a4"-"b4" "a5"-"b5"
 }
$$
\caption{Petersen graph $G$.}\label{petersen}
\end{figure}

Using
Procedure~\ref{procedure-example-petersen}, together with, Wei's
duality \cite[Theorem~3]{wei} we obtain Table~\ref{table2} describing the weight hierarchy of
$C_p(G)$. The edge biparticity, the weak edge
biparticity, and the edge connectivity of the Petersen graph are
equal to $3$.

\begin{table}[h]
\begin{eqnarray*}
\hspace{-11mm}&&\left.
\begin{array}{c|c|c|c|c|c|c|c|c|c|c}
r & 1 & 2 & 3& 4&5&6&7&8&9&10 \\
   \hline
\delta_r(C_2(G))& 3 & 5 & 7 & 9& 10&12 &13 &14 &15
 \\   
\hline  \delta_r(C_2(G)^\perp)
 &5 &8 &10  &12 &14 &15 & & & &\\
   \hline
\delta_r(C_3(G))& 3 & 5 & 7 & 8& 9&11 &12 &13 &14&15\\
\hline  \delta_r(C_3(G)^\perp)
 &6 &10 &12  &14 &15 & & & & &\\ 
\end{array}
\right.
\end{eqnarray*}
\caption{Weight hierarchy of $C_p(G)$ for the graph of
Figure~\ref{petersen}.}\label{table2}
\end{table}
\end{example}

\section{Procedures for Macaulay2 and
SageMath}\label{procedures-section}

\begin{procedure}\label{procedure-example-graph1}
Computing the weight hierarchies using 
{\it Macaulay\/}$2$ \cite{mac2}, SageMath \cite{sage}, and 
Wei's duality \cite{wei}. This procedure
corresponds to Example~\ref{example-graph1}. It could be 
applied to any connected graph $G$ to obtain the generalized Hamming weights 
of $C_p(G)$. The next procedure for {\it Macaulay\/}$2$ uses 
the algorithms of \cite{rth-footprint} to 
compute generalized minimum distance functions.

\begin{verbatim}
--Procedure for Macaulay2
input "points.m2"
q=3, R = ZZ/q[t1,t2,t3,t4,t5,t6]--p=char(K)=3
A = transpose(matrix{{1,1,0,0,0,0},{0,1,1,0,0,0},{1,0,1,0,0,0},
{0,0,0,1,1,0},{0,0,0,0,1,1},{0,0,0,1,0,1},{1,0,0,1,0,0},
{0,1,0,0,0,1},{0,0,1,0,1,0}})  
I=ideal(projectivePointsByIntersection(A,R)), M=coker gens gb I
genmd=(d,r)->degree M-max apply(apply(subsets(apply(apply(apply
(toList (set(0..q-1))^**(hilbertFunction(d,M))
-(set{0})^**(hilbertFunction(d,M)),toList),x->basis(d,M)*vector x),
z->ideal(flatten entries z)),r),ideal),x-> if #set flatten entries 
mingens ideal(leadTerm gens x)==r and not quotient(I,x)==I 
then degree(I+x) else 0)
--The following are the first two generalized Hamming weights
genmd(1,1), genmd(1,2)
\end{verbatim}

\begin{verbatim}
#Procedure for SageMath
A = transpose(matrix(GF(3),[[1,1,0,0,0,0],[0,1,1,0,0,0],[1,0,1,0,0,0],
[0,0,0,1,1,0],[0,0,0,0,1,1],[0,0,0,1,0,1],[1,0,0,1,0,0],[0,1,0,0,0,1],
[0,0,1,0,1,0]]))
C = codes.LinearCode(A)
C.parity_check_matrix()
C.generator_matrix()
#the next line Gives the minimum distance of the dual code
C.dual_code().minimum_distance()
\end{verbatim}
\end{procedure}

\begin{procedure}\label{procedure-example-petersen}
Computing the weight hierarchies and the edge biparticity using 
{\it Macaulay\/}$2$ \cite{mac2}, SageMath \cite{sage}, and 
Wei's duality \cite{wei}. This procedure
corresponds to Example~\ref{example-petersen}. The 
next procedure for {\it Macaulay\/}$2$ uses 
the algorithms of \cite{rth-footprint} to 
compute generalized footprint functions. The footprint gives 
easy to compute lower bounds for the generalized Hamming weights of
projective Reed--Muller-type codes. 
\begin{verbatim}
--Procedure for Macaulay2 for Petersen graph
input "points.m2"
R = QQ[t1,t2,t3,t4,t5,t6,t7,t8,t9,t10]
--Incidence matrix to compute the edge biparticity
A = transpose matrix{{1,1,0,0,0,0,0,0,0,0},{0,1,1,0,0,0,0,0,0,0},
{0,0,1,1,0,0,0,0,0,0},{0,0,0,1,1,0,0,0,0,0},{1,0,0,0,1,0,0,0,0,0},
{1,0,0,0,0,1,0,0,0,0},{0,1,0,0,0,0,1,0,0,0},{0,0,1,0,0,0,0,1,0,0},
{0,0,0,1,0,0,0,0,1,0},{0,0,0,0,1,0,0,0,0,1},{0,0,0,0,0,1,0,1,0,0},
{0,0,0,0,0,0,0,1,0,1},{0,0,0,0,0,0,1,0,0,1},{0,0,0,0,0,0,1,0,1,0},
{0,0,0,0,0,1,0,0,1,0}}
q=2, R = ZZ/q[t1,t2,t3,t4,t5,t6,t7,t8,t9]
--Generator matrix computed with Sage to find Hamming weights.
A1=matrix({{1,0,0,0,1,0,0,0,0,0,1,0,0,0,1},
{0,1,0,0,1,0,0,0,0,0,1,0,1,1,1},{0,0,1,0,1,0,0,0,0,0,0,1,1,1,1},
{0,0,0,1,1,0,0,0,0,0,0,1,1,0,0},{0,0,0,0,0,1,0,0,0,0,1,0,0,0,1},
{0,0,0,0,0,0,1,0,0,0,0,0,1,1,0},{0,0,0,0,0,0,0,1,0,0,1,1,0,0,0},
{0,0,0,0,0,0,0,0,1,0,0,0,0,1,1},{0,0,0,0,0,0,0,0,0,1,0,1,1,0,0}})
q=2, R = ZZ/q[t1,t2,t3,t4,t5,t6]
--Parity check matrix computed with Sage to find 
--the Hamming weights of dual code
A2=matrix({{1,0,0,0,0,1,1,0,0,0,0,0,0,1,1},
{0,1,0,0,0,0,1,1,0,0,0,1,1,0,0},{0,0,1,0,0,0,0,1,1,0,0,1,1,1,0},
{0,0,0,1,0,0,0,0,1,1,0,0,1,1,0},{0,0,0,0,1,1,0,0,0,1,0,0,1,1,1},
{0,0,0,0,0,0,0,0,0,0,1,1,1,1,1}})
--The following functions can be applied to A, A1, A2
I=ideal(projectivePointsByIntersection(A,R)), M=coker gens gb I
--This function computes the edge biparticity of Petersen graph 
--using the incidence matrix over the rational numbers
genmd1=(d,r)->degree M-max apply(apply(subsets(apply(apply(apply
(toList (set(1,-1))^**(hilbertFunction(d,M))
-(set{0})^**(hilbertFunction(d,M)),toList),x->basis(d,M)*vector x),
z->ideal(flatten entries z)),r),ideal),x-> if #set flatten entries 
mingens ideal(leadTerm gens x)==r and not quotient(I,x)==I 
then degree(I+x) else 0)
--To compute the r-th Hamming weight of the dual code 
--use genmd(1,r) of the previous procedure:
genmd(1,1),genmd(1,2),genmd(1,3),genmd(1,4),genmd(1,5)
--To compute the edge biparticity use genmd1(1,1)
init=ideal(leadTerm gens gb I), degree M
er=(x)-> if not quotient(init,x)==init then degree ideal(init,x) else 0
--This is the footprint function
fpr=(d,r)->degree M - max apply(apply(apply(subsets(flatten 
entries basis(d,M),r),toSequence),ideal),er)
--To find lower bounds for Hamming weights use the footprint:
fpr(1,1),fpr(1,2),fpr(1,3),fpr(1,4),fpr(1,5),fpr(1,6),fpr(1,7),fpr(1,8)
\end{verbatim}
\end{procedure}

\section*{Acknowledgments.} 
We thank the referees for a careful
reading of the paper and for the improvements suggested. 
Computations with \textit{Macaulay}$2$ \cite{mac2} and Sage \cite{sage}
were important to verifying examples given in this paper. 

\bibliographystyle{plain}

\end{document}